\documentclass[
graybox]{svmult}

\usepackage[obeyFinal,
textsize=scriptsize]{todonotes}

\usepackage[
  authormarkup=none]{changes}

\newcounter{ADD}
\newcounter{DEL}
\def\ADDin{\addtocounter{ADD}{1}}
\def\ADDout{\addtocounter{ADD}{-1}}
\def\DELin{\addtocounter{DEL}{1}}
\def\DELout{\addtocounter{DEL}{-1}}

\makeatletter
\newcommand{\ifdraft}[2]{}
\ifthenelse{\boolean{Changes@optiondraft}}{%
  \renewcommand{\ifdraft}[2]{#1}}{%
  \renewcommand{\ifdraft}[2]{#2}}
\newcommand{\ifADD}[2]{\ifnum\value{ADD}>0{#1}\else{#2}\fi}
\newcommand{\ifDEL}[2]{\ifnum\value{DEL}>0{#1}\else{#2}\fi}
\makeatother

\newcommand{\delmarkup}[1]{%
  \DELin%
  \textcolor{red!80!black}{\stkout{#1}}%
  \DELout%
}
\newcommand{\addmarkup}[1]{%
  \ADDin%
  \ifDEL{}{%
    \textcolor{green!50!black}{#1}}
  \ADDout%
}
\ifdraft{%
  \newenvironment{addedenv}{%
    \ADDin%
    \color{green!50!black}}
  {%
    \ADDout%
    \color{black}}
}
{
  
}
\ifdraft{
  \setdeletedmarkup{\delmarkup{#1}}
  \setaddedmarkup{\addmarkup{#1}}
}{}

\ifdraft{
  \newcommand{\deleteX}[1]{%
    \DELin%
    {\color{red!80!black}{\sout{#1}}}%
    \DELout}%
}{
  \newcommand{\deleteX}[1]{}
}

\ifdraft{}{%
  
}

\usepackage{type1cm}        
\usepackage{makeidx}         
\usepackage{graphicx}        
\usepackage{multicol}        
\usepackage[bottom]{footmisc}
\usepackage{todonotes}

\usepackage{newtxtext}       %
\usepackage{newtxmath}       
\usepackage[caption=false]{subfig}
\usepackage{tikz}
\usetikzlibrary{decorations.pathreplacing, patterns}
\usepackage{placeins}

\newcommand*{\ldblbrace}{\lbrace\mskip-5mu\lbrace}
\newcommand*{\rdblbrace}{\rbrace\mskip-5mu\rbrace}
\newcommand{\jump}[1]{\left\llbracket #1 \right\rrbracket}

\newcommand{\opext}{\mathcal{L}^{\operatorname{ext}}}

\newcommand{\average}[1]{\left\ldblbrace #1 \right\rdblbrace}

\newcommand{\scp}[2]{\left\langle #1, #2 \right\rangle}
\newcommand{\scpL}[2]{\left( #1, #2 \right)_{L^2}}
\newcommand{\Th}{%
  \mathcal{M}_h%
}
\newcommand{\ThBg}{{\widehat{\mathcal{M}}}_h}
\newcommand{\OmegaBg}{\widehat{\Omega}}
\newcommand{\dd}{\text{d}}
\newcommand{\GammaInt}{\Gamma_{\text{int}}}
\newcommand{\abso}[1]{\left\vert #1 \right\vert}
\newcommand{\norm}[1]{\left\lVert #1 \right\rVert}

\makeindex             

\begin{document}

\title*{DoD stabilization for higher-order advection in two dimensions}
\titlerunning{DoD stabilization for higher-order advection}
\author{Florian Streitb\"urger, Gunnar Birke, Christian Engwer, Sandra May}

\institute{Sandra May, Florian Streitb\"urger, \at TU Dortmund University, Vogelpothsweg 87, 44227 Dortmund, Germany, \email{florian.streitbuerger@math.tu-dortmund.de, sandra.may@math.tu-dortmund.de}
\and Christian Engwer, Gunnar Birke \at University of M\"unster, Einsteinstraße 62, 48149 M\"unster,  Germany, \email{christian.engwer@uni-muenster.de, g\_birk01@uni-muenster.de}
\and
This work was partially supported by the Deutsche Forschungsgemeinschaft (DFG, German Research Foundation) under the project number 439956613 (HyperCut) under the contracts {MA~7773/4-1} and EN~1042/5-1 and under Germany's Excellence Strategy EXC 2044-390685587, Mathematics Münster: Dynamics–Geometry–Structure.
}
%
%
\maketitle

\abstract{When solving time-dependent hyperbolic conservation laws on cut cell meshes one has to overcome the small cell problem: standard explicit time stepping is not stable on small cut cells if the time step is chosen with respect to larger background cells. The domain of dependence (DoD) stabilization is designed to solve this problem in a discontinuous Galerkin framework. It adds a penalty term to the space discretization that restores proper domains of dependency. In this contribution we introduce the DoD stabilization for solving the advection equation in 2d with higher order. We show an $L^2$ stability result for the stabilized semi-discrete scheme for arbitrary polynomial degrees $p$ and provide numerical results for convergence tests indicating orders of $p+1$ in the $L^1$ norm and between $p+\frac 1 2$ and $p+1$ in the $L^{\infty}$ norm.
}
\section{Introduction}
%
%
Modern simulations often require to mesh complex geometries. One approach that is particularly suited for this purpose are embedded boundary meshes. One simply cuts the geometry out of a structured background mesh, resulting in \textit{cut cells} along the boundary of the embedded object. Cut cells have different shapes and can become arbitrarily small. In the context of solving time-dependent hyperbolic conservation laws this causes the \textit{small cell problem}: for standard explicit time stepping, the scheme is not stable on small cut cells when the time step is chosen with respect to the larger background cells. 

Existing solution approaches in a finite volume regime are typically bound to at most second order, see for example the flux redistribution method \cite{Chern_Colella,Colella2006}, the $h$-box method \cite{Berger_Helzel_Leveque_2005,Berger_Helzel_2012}, the mixed explicit-implicit scheme \cite{May_Berger_explimpl}, the dimensionally split approach \cite{Klein_cutcell,Klein_cutcell_3d}, or the state redistribution (SRD) method \cite{Berger_Giuliani_2021}. An exception is the extension of the active flux method to cut cell meshes \cite{FVCA_Helzel_Kerkmann}, which aims for third order. 

For discontinuous Galerkin (DG) schemes it is significantly easier to achieve higher order. The development of DG schemes that overcome the small cell problem has only started very recently. Some work relies on cell merging, e.g. \cite{Kronbichler2020}, other work on algorithmic solution approaches such as the usage of a ghost penalty term as done by Fu and Kreiss \cite{Kreiss_Fu} or the extension of the SRD method to a DG setting by Giuliani \cite{Giuliani_DG}. 

Another approach, proposed previously by the authors, is the \textit{Domain-of-Dependence} (DoD) stabilization. In this approach, a penalty term is added on small cut cells that restores the proper domains of dependence in the neighborhood of small cut cells and therefore makes standard explicit time stepping stable again. In \cite{DoD_SIAM_2020} we first introduced the DoD stabilization for linear advection in 1d and 2d for linear polynomials only. In \cite{DoD_AMC_2021}, we extend the stabilization in 1d to non-linear systems and higher order. For the extension to higher order in 1d we found that it is necessary to add an extra term in the stabilization, which adjusts the mass distribution within inflow neighbors of small cut cells. With this term it is possible to show an $L^2$ stability result for the semi-discrete setting (keeping the time continuous) in 1d \cite{DoD_AMC_2021}. 

In this contribution, we partially extend the 1d results from \cite{DoD_AMC_2021} to 2d by solving linear advection with higher order polynomials. For the case of a planar ramp geometry we show an $L^2$ stability result for the semi-discrete setting. We will also provide corresponding numerical results. These results show the expected convergence orders of $p+1$ for polynomial degree $p$ in the $L^1$ norm. In the $L^{\infty}$ norm, we observe a slight decay, resulting in convergence orders between $p+\frac 1 2$ and $p+1$.

\section{Problem setup}
Within the scope of this work, we will focus on the 2d linear advection equation
\begin{align}\label{eq: lin adv 2d}
      u_t+\scp{\beta}{\nabla u}&=0&\quad&\text{in }\Omega \times (0,T),\\
  u &=g & &\text{on }\partial\Omega^{\text{in}} \times (0,T),\\
  u &= u_0 & &\text{on }\Omega \times \{ t = 0 \}.
\end{align}
We denote by $\Omega$ an open, connected domain in $\mathbb{R}^2$ and by $\partial \Omega$ its boundary. The inflow boundary is defined as $\partial\Omega^\text{in}:=\{x\in \partial \Omega : \scp{\beta(x)}{n(x)}<0\}$ with $n\in \mathbb{R}^2$ being the outer unit normal vector on $\partial \Omega$. Analogously, we define
$\partial\Omega^\text{out}:=\{x\in \partial \Omega : \scp{\beta(x)}{n(x)}>0\}$.
Moreover, $\beta\in\mathbb{R}^2$ is the velocity field and $\scp{\cdot}{\cdot}$ the standard scalar product in $\mathbb{R}^2$. 

For simplicity and brevity of presentation, we will focus in this presentation on the case of a ramp geometry with a constant velocity field $\beta \in \mathbb{R}^2$, which is parallel to the ramp. The geometry setup and mesh creation is sketched in figure \ref{fig:mesh-geom-intersect}.
We refer to the internal and external faces of our mesh $\Th$ as
\begin{align}
  \GammaInt &=
  \left\{ e_{E_1,E_2} = \partial E_1 \cap \partial E_2
    \ \left|\  E_1, E_2 \in \Th
    \ \text{and}\  E_1 \neq E_2
    \ \text{and}\ \vert e_{E_1,E_2}\vert >0\right.\right\}, 
  \label{eq:int_skel} \\
  \Gamma_{\text{ext}} &=
  \left\{ e_{E} = \partial E \cap \partial\Omega
    \ \left|\ E \in \Th
    \ \text{and}\ \vert e_E \vert >0\right.\right\},
    \label{eq:ext_skel}
    \end{align}
    with $\vert e \vert$ denoting the length of an edge $e$. We then further split $\Gamma_{\text{ext}}$ in $\Gamma_{\text{ext,Cart}}$ and $\Gamma_{\text{ext,ramp}}$: $\Gamma_{\text{ext,Cart}}$ contains all Cartesian boundary faces and $\Gamma_{\text{ext,ramp}}$ contains the boundary faces along the ramp that were created by cutting out the ramp geometry. 

\begin{figure}[tp]
  \centering
  \begin{tikzpicture}[scale=1.8]
    \node at (.75,1.15) {\footnotesize$\ThBg$};
    \draw[semithick,step=0.25] (0,0) grid (1.5,1);
    \node at (1.9,0.5) {\huge$\cap$};
\begin{scope}[xshift=2.3cm]
    \node at (.75,1.15) {\footnotesize$\overline{\Omega}$};
    \draw[semithick,fill=black!10!white]
    (0.0,0.0) --
    (0.3,0.0) --
    (1.5,0.7) --
    (1.5,1.0) --
    (0.0,1.0) --
    (0.0,0.0);
    \draw[dashed](0.3,0) -- +(0:0.7) arc (-45:45:0.29) -- cycle;
    \node[] at (0.9,0.15) {\large $\gamma$} ;
\end{scope}
    \node at (4.2,0.5) {\huge$\rightarrow$};
\begin{scope}[xshift=4.6cm]
    \node at (.75,1.15) {\footnotesize$\Th$};
    \draw[semithick,step=0.25] (0,0) grid (1.5,1);
    \fill[white] (0.3,0.0) --
      (1.5,0.7) --
      (1.51,0.7) --
      (1.51,-0.01) --
      (0.3,-0.01);
    \draw[semithick]
    (0.0,0.0) --
    (0.3,0.0) --
    (1.5,0.7) --
    (1.5,1.0) --
    (0.0,1.0) --
    (0.0,0.0);
\end{scope}
\end{tikzpicture}
  \caption{Construction of cut cell mesh $\Th$ for the case of a ramp geometry: We intersect the structured background mesh $\ThBg$ of a larger rectangular domain $\OmegaBg \supset \Omega$ with the domain $\overline{\Omega}$. Here, $\Omega$ has a ramp geometry described by the angle $\gamma$. This results in cut cells $E = \widehat E \cap \overline{\Omega}$ along the ramp, where $\widehat
    E \in \ThBg$ is an element of the background mesh.}
  \label{fig:mesh-geom-intersect}
\end{figure}
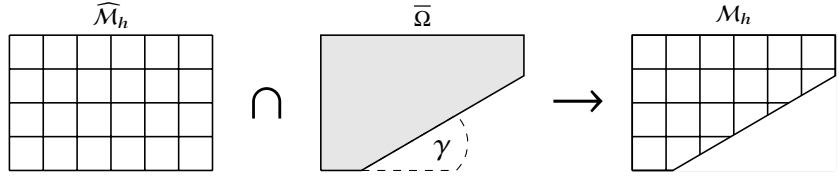

On the partition $\Th$, we define the discrete function space
$\mathcal{V}_h^p(\Th) \subset L^2(\Omega)$ by
\begin{equation}\label{eq: def V_h}
   \mathcal{V}_h^p(\Th) = \left\{ v_h \in L^2(\Omega)  \: \vert \: \forall E \in \Th, v_h{\vert_E} \in P^p(E) \right\},
\end{equation}
where $P^p$ denotes the polynomial space of degree $p$. 

On a face $e$ between two adjacent cells $E_1$ and $E_2$, i.e., $e = \partial{E_1}\cap\partial{E_2}$, we define the scalar-valued 
\emph{average} as
  \begin{align*}
    \average{u_h}=\frac 1 2({u_h}\vert_{_{E_1}}+{u_h}\vert_{_{E_2}}),
  \end{align*}
  and the \emph{jump} to be vector-valued as
  \begin{align}\label{eq: jump in 2d}
    \jump{u_h}:={u_h}\vert_{_{E_1}}n_{E_1}+{u_h}\vert_{_{E_2}}n_{E_2},
  \end{align}
  with $n_{E_i}$ denoting the outer unit normal vector of cell $E_i$, $i=1,2$. With these prerequisites we can now define the scheme that we use to solve \eqref{eq: lin adv 2d}.

We use a method of lines approach: we first discretize in space and then in time. The unstabilized semi-discrete scheme is given by: Find $u_h(t) \in  \mathcal{V}^p_h(\Th)$ such that
\begin{align}\label{eq: scheme 2d wo stab}
  \scpL{d_tu_h(t)}{w_h}+a_h^{\text{upw}}\!\left(u_h(t), w_h\right)+l_h\left(w_h\right) = 0\quad\forall w_h\in \mathcal{V}_h^p(\Th),
\end{align}
with
\begin{align}
  &\begin{aligned}
      a_h^{\text{upw}}(u_h, w_h):=&-\sum_{E \in \Th} \int_E u_h\scp{\beta}{\nabla_hw_h}\dd{x} + \sum_{e \in \Gamma_{\text{ext}}} \int_{e} \scp{\beta}{n}^{\oplus}u_hw_h \dd{s}\\
      & +\sum_{e \in \GammaInt}\int_{e} \left( \average{u_h}\scp{\beta}{\jump{w_h}}
    + \frac{1}{2}\abso{\scp{\beta}{n_e}}\scp{\jump{u_h}}{\jump{w_h}} \right) \dd{s},
    \end{aligned}
    \label{eq:aupw}\\
    &\begin{aligned}
      l_h(w_h):=&-\sum_{e \in \Gamma_{\text{ext}}} \int_{e}\scp{\beta}{n}^\ominus g \: w_h\dd{s}.\label{eq: l_h 2d}
    \end{aligned}
\end{align}
Here, $\scpL{\cdot}{\cdot}$ denotes the standard scalar product in $L^2(\Omega)$ and $n_e\in\mathbb{R}^2$ is a unit normal on a face $e$ (of arbitrary but fixed orientation). We define the negative and positive parts of 
$x\in\mathbb{R}$ as
$
  x^\ominus:=\frac{\abso{x}-x}{2}$ and $x^\oplus:=\frac{\abso{x}+x}{2}.
 $ 
Note that that the standard upwind flux is used in the definition of $a_h^{\text{upw}}$. 

The proposed stabilization modifies the \emph{space} discretization, whereas in time we are free to use a time stepping scheme of our choice. We will use standard explicit strong stability preserving (SSP) Runge Kutta (RK) schemes \cite{GottliebShuTadmor}.

\section{Stabilization terms}
    The stabilization is designed as an additional term $J_h$ that is added to the semi-discrete formulation in \eqref{eq: scheme 2d wo stab}. The DoD stabilized semi-discrete scheme is then given by: Find $u_h(t) \in \mathcal{V}^p_h(\Th)$ such that
    \begin{equation}\label{eq:stabilized scheme}
    \begin{split}
        \scpL{d_tu_h(t)}{w_h} + a_h^{\text{upw}}\!\left(u_h(t), w_h\right)
        &+J_h(u_h(t),w_h)\\
        &+l_h\left(w_h\right) = 0\qquad\forall w_h\in \mathcal{V}_h^p(\Th).
    \end{split}
    \end{equation}
The stabilization term $J_h$ is given by
\[
    J_h(u_h,w_h) = J_h^0(u_h,w_h)+J_h^1(u_h,w_h)= \sum_{E \in \mathcal{I}} \left( J_h^{0,E}(u_h,w_h) + J_h^{1,E}(u_h,w_h) \right).
\]
We define $J_h^{0,E}$ and $J_h^{1,E}$ in detail below. 
The set $\mathcal{I}$ denotes the set of small cut cells that need stabilization.
For a planar cut in 2d, there are 3-sided, 4-sided, and 5-sided cut cells. In \cite{DoD_SIAM_2020}, we have shown  (see Lemma 3.5), that for the considered setup, it is sufficient to stabilize \textit{triangular cut cells} only. For a triangular cut cell $E_{\text{cut}}$ in our setup, each edge has a different boundary condition, see figure \ref{fig: triangular cell}:
\begin{itemize}
    \item On the \textit{boundary edge} $e_{\text{bdy}}$ we have a no-flow boundary condition as the flow is parallel to the ramp.
    \item Out of the two remaining edges, one edge is the \textit{inflow edge} $e_\text{in}$, which is
    characterized by $\scp{\beta}{n_{E_{\text{cut}}}}<0$. 
    \item The remaining edge is the \textit{outflow edge} $e_\text{out}$, which is characterized by $\scp{\beta}{n_{E_{\text{cut}}}}\geq 0$.
\end{itemize}
%
Thus we can uniquely define an inflow neighbor $E_\text{in}$ and an outflow neighbor
$E_\text{out}$ for a triangular cut cell $E_{\text{cut}}$.

\begin{figure}[t]
    \centering
    \begin{tikzpicture}[color=black,semithick,scale=1.4]
\fill[pattern=north east lines, pattern color=black!50!white]  (-0.8,-1) -- (1,0.8) -- (1,-1) -- (-0.8,-1);
\fill[color=black!10!white]  (-0.5,0.5) -- (0.7,0.5) -- (-0.5,-0.7) -- (-0.5,0.5);
\draw[] (-1,0.5) -- (0.7,0.5);
\draw[] (-0.5,1) -- (-0.5,-0.7);
\draw[very thick] (-0.8,-1) -- (1,0.8);
\draw (-0.5,0.5) rectangle (1,-1);
\draw (1,-1) -- (-1.,-1);
\draw (1,-1) -- (1,1);
\draw[->, very thick] (-1.5,-0.2) -- (-1,0.3);
\node[anchor=east] at (-1.2,0.2) {${\beta}$};
\node[anchor=east] at (-0.5,-0.125) {${E_\text{in}}$};
\node[anchor=south] at (0.125,0.5) {${E_\text{out}}$};
\node[anchor=south] at (-0.1,0) {${E_\text{cut}}$};
\node[anchor=west] at (1.3,-0.3) {$e_\text{bdy}$}; 
\node[anchor=west] at (-0.5,-0.125) {$e_\text{in}$};
\draw[->, very thick] (1.3,-0.3) -- (0.2,-0.1);
\node[anchor=north] at (0.125,0.5) {$e_\text{out}$};
\end{tikzpicture}
\caption{Triangular cut cell}
\label{fig: triangular cell}
\end{figure}
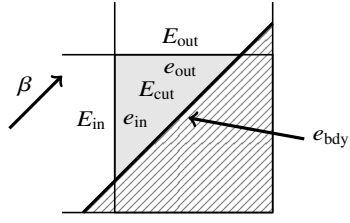

If the time step is chosen according to the size of the larger background cells and does not respect the size of small cut cells, then, physically, mass passes within one time step from the inflow cell $E_{\text{in}}$ through the small cut cell $E_{\text{cut}}$ into the outflow cell $E_{\text{out}}$. The idea behind the DoD stabilization is to make this possible by directly passing part of the mass that enters $E_{\text{cut}}$ from $E_{\text{in}}$ into the outflow neighbor $E_{\text{out}}$. This way we restore the domain of dependence of the outflow neighbor $E_{\text{out}}$ and make sure that the small cut cell $E_{\text{cut}}$ only keeps as much mass as it can hold. For the latter one we define the concept of \textit{capacity} below.

    In order to create this flux of information between the inflow neighbor and the outflow neighbor of $E_{\text{cut}}$, we introduce an extension operator: The operator $\opext_{E'}$
  extends a function $u_h \in \mathcal{V}^p_h$ from a cell $E'\in\Th$ to the whole
  domain $\Omega$. This simply corresponds to evaluating a polynomial outside its original support. In particular, we will evaluate the polynomial solution defined on cell $E_{\text{in}}$ to $\overline{E_{\text{cut}}}$. We will refer to this both as $\opext_{E_\text{in}}(u_h)(x)$,
  $x \in \overline{E_{\text{cut}}}$, as well as simply as $u_{E_{\text{in}}}(x)$ to ease notation.
    
    With these prerequisites we can now define $J_h^0$ and $J_h^1$. 
Generally, the terms of the DoD stabilization
target two different goals:
    \begin{itemize}
        \item $J_h^0$ aims for redistributing the mass \textit{among} the small cut cells and their neighbors appropriately. It therefore consists of \textit{cell interface} terms.
        \item $J_h^1$ aims for redistributing the mass \textit{within} the small cut cells and their neighbors appropriately. It therefore consists of \textit{volume} terms.
    \end{itemize}
The term $J_h^{0,E}$ is given by
\begin{equation}
    J_h^{0,E}(u_h,w_h) = \eta_E \int_{e_\text{out}} (\opext_{E_\text{in}}(u_h)-u_h) \scp{\beta}{\jump{w_h}}\dd s,
\end{equation}
with the stabilization parameter $\eta_E$ defined below. Note that we only redistribute mass across outflow edges of small cut cells and that we use the extended solution of the inflow neighbor to determine the size of the correction. 
The term $J_h^{1,E}$ is given by
\begin{equation}
    J_h^{1,E}(u_h,w_h) = \eta_E \int_E (\opext_{E_\text{in}}(u_h)-u_h) \scp{\beta}{\opext_{E_\text{in}}(\nabla w_h)-\nabla w_h} \dd x.
\end{equation}
This term is designed to adjust the mass distribution primarily within the small cut cell $E$ and secondarily within its neighbor. Note that we apply the extension operator to both the discrete solution and the test function from inflow neighbor $E_{\text{in}}$. In \cite{DoD_SIAM_2020}, where we only considered piecewise linear polynomials, we proposed a different formulation of $J_h^{1,E}$, which did not contain the term $\opext_{E_\text{in}}(\nabla w_h)$. In 1d \cite{DoD_AMC_2021}, we found that when going to higher order one can run into instabilities without this extra term. In addition, with this augmented definition of $J_h^{1,E}$ we are able to show an $L^2$ stability result for the semi-discrete scheme, which we will present below.

Both stabilization terms are scaled with the stabilization parameter $\eta_E$. We set $\eta_E = 1-\alpha_{E,1/(2p+1)}$ with the capacity $\alpha_{E,\omega}$ and $p$ being the polynomial degree of the discrete function space.  We define the capacity of a cut cell $E$, see \cite{DoD_SIAM_2020}, as
  \begin{align}\label{eq:def:alphaE:omega}
    \alpha_{E,\omega}:= \min\left(\omega \frac{\abso{E}}{\Delta t \int_{\partial E}\scp{\beta}{n_{E}}^\ominus\dd{s}},1\right), \quad \omega \in (0,1].
  \end{align}
  For $\omega=1$, the capacity estimates the fraction of the inflow
  that is allowed to flow into the cut cell $E$ and stay there without producing overshoot.
  Note that by definition $0 \le \eta_E \le 1$.

\subsection{$L^2$ stability for semi-discrete scheme}
In the following we will show an $L^2$ stability result for the stabilized semi-discrete scheme for an arbitrary polynomial degree $p$. Generally, the $L^2$ stability result for the considered ramp setup is influenced by the inflow and outflow across $\partial\Omega^{\text{in}}$ and $\partial\Omega^{\text{out}}$ during the time $(0,T)$. Note that only Cartesian faces $e \in \Gamma_{\text{ext,Cart}}$ are contained in $\partial\Omega^{\text{in}} \cup \partial\Omega^{\text{out}}$ as we have a no-flow boundary condition for faces
$e \in \Gamma_{\text{ext,ramp}}$ along the ramp. 
Our goal here is to show that $L^2$ stability still holds true for the stabilized scheme \textit{with cut cells being present}, and not to analyze the influence of the inflow and outflow on the $L^2$ stability. We will therefore for simplicity assume that the solution has compact support inside $\hat{\Omega}$
 during the considered time interval $(0,T)$ and does not intersect the Cartesian boundary, i.e. $\text{supp}(u)\cap (\partial\Omega^{\text{in}} \cup \partial\Omega^{\text{out}}) = \emptyset$, which implies that we have a homogeneous right hand side during the whole time frame~$(0,T)$ and in particular that
 there is no in- or outflow.
%

\begin{theorem}\label{theorem: l2 stability}
Consider the advection equation \eqref{eq: lin adv 2d} for the setup of a ramp with constant velocity field $\beta = (\beta_1,\beta_2)^T$. Let the solution $u$ have compact support during the considered time interval $(0,T)$ and $\text{supp}(u)\cap (\partial\Omega^{\text{in}} \cup \partial\Omega^{\text{out}}) = \emptyset$. Let $u_h(t)$, with $u_h(t) \in \mathcal{V}^p_h$ for any fixed $t$, be the solution to the stabilized semi-discrete problem \eqref{eq:stabilized scheme}. 
Then, the solution satisfies for all $t\in (0,T)$
\begin{equation*}
    \norm{u_h(t)}_{L^2(\Omega)} \leq \norm{u_h(0)}_{L^2(\Omega)}.
\end{equation*}
\end{theorem}

\begin{proof}
Setting $w_h = u_h(t)$ in \eqref{eq:stabilized scheme} and ignoring boundary contributions with respect to $\partial\Omega^{\text{in}}$, we get
\begin{equation*}
\scpL{d_tu_h(t)}{u_h(t)}+a_h^{\text{upw}}\left(u_h(t), u_h(t)\right)+J_h(u_h(t),u_h(t)) = 0.
\end{equation*}
Integration of the first term in time yields
\begin{align*}
\int_{0}^t \scpL{d_{\tau} u_h(\tau)}{u_h(\tau)} \: \dd \tau &= 
\int_0^t \frac{d}{d\tau} \frac{1}{2} \norm{u_h(\tau)}_{L^2(\Omega)}^2 \dd \tau \\
&= \frac 1 2 \norm{u_h(t)}_{L^2(\Omega)}^2 - \frac 1 2  \norm{u_h(0)}_{L^2(\Omega)}^2,
\end{align*}
and it remains to show that for any fixed $t$
\begin{equation*}
a_h^{\text{upw}}(u_h(t), u_h(t)) +  J_h(u_h(t),u_h(t)) \ge 0.
\end{equation*}
We will first discuss $a_h^{\text{upw}}$ and then $J_h$. (We will drop the explicit time dependency in the following for brevity.)

By definition of $a_h^{\text{upw}}$ and ignoring outflow across $\partial\Omega^{\text{out}}$, there holds 
\begin{align*}
      a_h^{\text{upw}}(u_h, u_h)=&-\sum_{E \in \Th} \int_E u_h\scp{\beta}{\nabla_h u_h}\dd{x} \\
      & +\sum_{e \in \GammaInt}\int_{e} \left( \average{u_h}\scp{\beta}{\jump{u_h}}
    + \frac{1}{2}\abso{\scp{\beta}{n_e}}\scp{\jump{u_h}}{\jump{u_h}} \right) \dd{s}.
\end{align*}
For the integral term we rewrite 
\[
- \int_E u_h\scp{\beta}{\nabla_h u_h}\dd{x} = 
- \int_E \nabla \cdot \left(\frac 1 2 \beta u_h^2 \right) \dd{x}
= - \int_{\partial E} \left( \frac 1 2 \beta u_h^2 \right) \cdot n \: \dd{s}.
\]
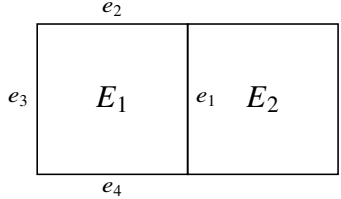
\begin{figure}[b]
    \centering
    \begin{tikzpicture}[color=black,semithick,scale=1.0]
\draw (-2,-1) rectangle (0,1);
\draw (0,-1) rectangle (2,1);
\node[] at (-1,0) {\large $E_1$};
\node[] at (1,0) {\large $E_2$};
\node[anchor=west] at (0,0) {$e_1$};
\node[anchor=south] at (-1,1) {$e_2$};
\node[anchor=east] at (-2,0) {$e_3$};
\node[anchor=north] at (-1,-1) {$e_4$};
\end{tikzpicture}
%
%
\caption{Setup for Cartesian cells}
\label{fig: cells l2 stab}
\end{figure}%
Let us first consider a standard Cartesian cell $E_1$ with edges as shown in figure \ref{fig: cells l2 stab}. Then, for $\beta = (\beta_1,\beta_2)^T$, there holds
\[
- \int_{E_1}\!\!u_h\scp{\beta}{\nabla_h u_h}\dd{x} 
= - \int_{e_1} \!\frac 1 2  \beta_1u_h^2 \: \dd{s} -  \int_{e_2}\!\frac 1 2  \beta_2 u_h^2 \: \dd{s}
+ \int_{e_3}\!\frac 1 2  \beta_1 u_h^2 \: \dd{s} +  \int_{e_4}\!\frac 1 2 \beta_2 u_h^2 \: \dd{s}.
\]
For the edge terms in $a_h^{\text{upw}}$, let us consider the internal edge $e_1$, connecting $E_1$ and $E_2$. Then, (using from now on the notation $u_{E'}$ to indicate that we evaluate the discrete solution from cell $E'$, potentially outside of its original support)
\begin{align*}
    \int_{e_1} &\left( \average{u_h}\scp{\beta}{\jump{u_h}}
    + \frac{1}{2}\abso{\scp{\beta}{n_{e_1}}}\scp{\jump{u_h}}{\jump{u_h}} \right) \dd{s}\\
    &= \int_{e_1} \left(\frac 1 2 \beta_1({u_{E_1}}+{u_{E_2}}) ({u_{E_1}}- {u_{E_2}}) + \frac 1 2 \beta_1({u_{E_1}}- {u_{E_2}})^2 \right) \dd{s} \\ 
    &= \int_{e_1} \beta_1\left( ({u_{E_1}})^2 - {u_{E_1}} {u_{E_2}} \right) \dd{s}.
\end{align*}%
Combining this with the corresponding contributions for edge $e_1$ from the volume terms from cells $E_1$ and $E_2$, we get
\begin{align*}
   - \int_{e_1} \frac 1 2 \beta_1({u_{E_1}})^2 \: \dd{s} 
   +& \int_{e_1} \frac 1 2 \beta_1({u_{E_2}})^2 \: \dd{s}\\
   +& \int_{e_1} \left( \average{u_h}\scp{\beta}{\jump{u_h}}
   + \frac{1}{2}\abso{\scp{\beta}{n_{e_1}}}\scp{\jump{u_h}}{\jump{u_h}} \right) \dd{s}\\
   &= \int_{e_1} \left( \frac 1 2 \beta_1({u_{E_1}})^2 - \beta_1{u_{E_1}}{u_{E_2}} + \frac 1 2 \beta_1({u_{E_2}})^2\right) \dd{s} \\
   &= \int_{e_1} \frac 1 2 \beta_1\left( {u_{E_1}} - {u_{E_2}} \right)^2 \dd{s}.
\end{align*}%
%
%
%
Let us now add the cut cells. For the small triangular cut cell $E_\text{cut}$ with the notation from figure \ref{fig: triangular cell}, we get with $\beta=(\beta_1,\beta_1)^T$
\[
- \int_{E_\text{cut}} u_h\scp{\beta}{\nabla_h u_h}\dd{x} 
= - \int_{\partial E_\text{cut}} \left( \frac 1 2 \beta u_h^2 \right) \cdot n \: \dd{s}
= - \int_{e_\text{out}}  \frac 1 2 \beta_2 u_h^2  \: \dd{s} + \int_{e_\text{in}}  \frac 1 2 \beta_1 u_h^2  \: \dd{s}.
\]%
Therefore, taking the boundary term in $a_h^{\text{upw}}$ into account as well as the contribution from the volume term of cell $E_{\text{in}}$, we get for the edge $e_{\text{in}}$
\begin{align*}
   - \int_{e_\text{in}} \frac 1 2  \beta_1({u_{E_\text{in}}})^2 \: \dd{s} 
   +& \int_{e_\text{in}} \frac 1 2 \beta_1({u_{E_\text{cut}}})^2 \: \dd{s}\\
   +& \int_{e_\text{in}} \left( \average{u_h}\scp{\beta}{\jump{u_h}}
   + \frac{1}{2}\abso{\scp{\beta}{n_{e_\text{in}}}}\scp{\jump{u_h}}{\jump{u_h}} \right) \dd{s}\\
   &= \int_{e_\text{in}} \left( \frac 1 2 \beta_1({u_{E_\text{in}}})^2 - \beta_1{u_{E_\text{in}}} {u_{E_\text{cut}}} + \frac 1 2 \beta_1({u_{E_\text{cut}}})^2\right) \dd{s} \\
   &= \int_{e_\text{in}} \frac 1 2 \beta_1\left( {u_{E_\text{in}}} - {u_{E_\text{cut}}} \right)^2 \dd{s}.
\end{align*}%
We obtain a similar term for edge $e_\text{out}$, involving solutions from cells $E_\text{cut}$ and $E_\text{out}$. 
Therefore, ignoring boundary contributions across $\partial\Omega^{\text{in}}\cup\partial\Omega^{\text{out}}$ due to the assumption of compact support, there holds
\begin{equation}\label{eq: L2 stab a_h}
a_h^{\text{upw}}(u_h, u_h) = \sum_{e \in \GammaInt} \frac 1 2 \int_e \abso{\scp{\beta}{n_e}}
\scp{\jump{u_h}}{\jump{u_h}} \: \dd{s}.
\end{equation}
Therefore, without the stabilization term $J_h$, there holds $L^2$ stability.

Let us now add the stabilization term
\begin{align*}
    J_h(u_h,u_h) &= \sum_{E \in \mathcal{I}} J_h^{0,E}(u_h,u_h) + J_h^{1,E}(u_h,u_h).
\end{align*}
We only stabilize small triangular cells of type $E_\text{cut}$. There holds
\begin{align*}
    J_h^{0,E_\text{cut}}&(u_h,u_h) = \eta_{E_\text{cut}} \int_{e_{\text{out}}} ({u_{E_{\text{in}}}}-{u_{E_\text{cut}}}) \scp{\beta}{\jump{u_h}}\dd s \\
    &= \eta_{E_\text{cut}} \int_{e_{\text{out}}}\beta_2 ({u_{E_{\text{in}}}}-{u_{E_\text{cut}}})({u_{E_\text{cut}}}-{u_{E_{\text{out}}}}) \dd s \\
    &= \eta_{E_\text{cut}} \int_{e_\text{out}} \beta_2
    \left(
    {u_{E_\text{in}}}{u_{E_\text{cut}}} - 
    {u_{E_\text{in}}} {u_{E_\text{out}}} -
    ({u_{E_\text{cut}}})^2 + {u_{E_\text{cut}}}{u_{E_\text{out}}}
    \right)\dd s.
\end{align*}
We now consider $J^{1,E_\text{cut}}$ given by
\[
    J_h^{1,E_\text{cut}}(u_h,u_h) = \eta_{E_\text{cut}} \int_{E_\text{cut}} (u_{E_\text{in}}-u_{E_\text{cut}}) \scp{\beta}{\nabla u_{E_\text{in}}-\nabla u_{E_\text{cut}}} \dd x.
\]
With $\beta = (\beta_1,\beta_2)^T$, there holds
\begin{align*}
J_h^{1,E_\text{cut}}&(u_h,u_h) = \eta_{E_\text{cut}} \int_{E_\text{cut}} \nabla \cdot \left( \frac 1 2 \beta (u_{E_\text{in}}-u_{E_\text{cut}})^2 \right) \dd x \\
&= \eta_{E_\text{cut}} \int_{\partial E_\text{cut}} \left( \frac 1 2 \beta (u_{E_\text{in}}-u_{E_\text{cut}})^2 \right) \cdot n \: \dd{s} \\
&= \eta_{E_\text{cut}} \int_{e_\text{out}} \left( \frac 1 2 \beta_2 (u_{E_\text{in}}-u_{E_\text{cut}})^2 \right)\:\dd{s}
- \eta_{E_\text{cut}} \int_{e_\text{in}} \left( \frac 1 2 \beta_1 (u_{E_\text{in}}-u_{E_\text{cut}})^2 \right)\:\dd{s}.
\end{align*}
As $0 \le \eta_{E_\text{cut}} \le 1$, the negative term over the edge $e_\text{in}$ can be compensated with the edge term $\int_{e_\text{in}}\beta_1 \left( \frac 1 2 (u_{E_\text{in}}-u_{E_\text{cut}})^2 \right)\:\dd{s} $ from $a_h^{\text{upw}}$ in \eqref{eq: L2 stab a_h}. For the edge $e_\text{out}$, we collect all terms from $J^{0,E_\text{cut}}$ and $J^{1,E_\text{cut}}$ to get
\begin{align*}
&\eta_{E_\text{cut}} \int_{e_\text{out}} \beta_2\bigg( 
{u_{E_\text{in}}}{u_{E_\text{cut}}} - 
    {u_{E_\text{in}}}{u_{E_\text{out}}} -
    ({u_{E_\text{cut}}})^2 + {u_{E_\text{cut}}}{u_{E_\text{out}}}
    + \frac 1 2 (u_{E_\text{in}}-u_{E_\text{cut}})^2
\bigg) \: \dd{s}\\
&= \eta_{E_\text{cut}} \int_{e_\text{out}} \beta_2
\left( 
\frac 1 2 (u_{E_\text{in}})^2 - \frac 1 2 (u_{E_\text{cut}})^2
-  {u_{E_\text{in}}} {u_{E_\text{out}}} + {u_{E_\text{cut}}}{u_{E_\text{out}}}
\right) \: \dd{s} \\
&= \eta_{E_\text{cut}} \int_{e_\text{out}} 
\frac 1 2 \beta_2 (u_{E_\text{in}} - {u_{E_\text{out}}})^2  \: \dd{s}
- \eta_E \int_{e_\text{out}} 
\frac 1 2 \beta_2 (u_{E_\text{cut}} - {u_{E_\text{out}}})^2 \: \dd{s}.
\end{align*}
The right term in the last line involves the standard jump over edge $e_\text{out}$ and (same as for edge $e_\text{in}$) can be compensated with its positive counterpart in the sum in \eqref{eq: L2 stab a_h}. The first term in the last line consists of a new extended jump involving the difference of the solution
of cell $E_{\text{in}}$ and the solution of cell $E_{\text{out}}$, both evaluated on the outflow edge $e_\text{out}$.
This concludes the proof.
\end{proof}

\section{Numerical results}

In this section, we present numerical results for the linear advection equation in 2d using higher order polynomials for the ramp setup introduced above for different angles $\gamma$, see figure \ref{fig:mesh-geom-intersect}. We choose $\hat{\Omega} = (0,1)^2$ and start the ramp at
$x=0.2001$.
For the definition of the initial data, we use a rotated and shifted coordinate system $(\hat{x},\hat{y})$ that we derive from the standard Cartesian coordinate system $(x,y)$ by
\begin{equation}
    \begin{pmatrix} \hat{x} \\ \hat{y} \end{pmatrix} = \begin{pmatrix}\cos{\gamma}& \sin{\gamma}\\
-\sin{\gamma}&\cos{\gamma}\end{pmatrix}\cdot
\begin{pmatrix}x-0.2001\\
y\end{pmatrix}.
\end{equation}
This newly described coordinate system is defined in such a way that the $\hat{x}$-direction is parallel and the $\hat{y}$-direction is orthogonal to the ramp.  
In this coordinate system, the velocity field $\beta$ and the smooth initial data are given by
\[
\beta(\hat{x},\hat{y}) = 2\begin{pmatrix} 1\\ 0\end{pmatrix}, \quad 
u_0(\hat{x},\hat{y})= \sin\left(\frac{\sqrt{2}\pi \hat{x}} {1-0.2001}\right).
\]
We derive the inflow conditions on $\partial\Omega_{\text{in}}$ from the exact solution. We compute the discrete solution at time $T=0.3$ using polynomials of degrees $p=1,2,3$. In time we use an SSP RK scheme of the same order as the space discretization. We compute the time step $\Delta t$ by
\begin{equation}\label{eq: dt in 2d}
\Delta t \le 0.4 \frac{1}{2p+1} \frac{h}{\lVert\beta\rVert}.
\end{equation}
Here, $h = 1/N$ with $N$ being the number of cells in $x$- and $y$-direction on
$\hat{\Omega}$.

The implementation is based on the DUNE \cite{dune08:1,dune08:2}
framework, the cut-cell DG extension \texttt{dune-udg}
package \cite{duneudg,Bastian_Engwer} and its integration with
\texttt{dune-pdelab}. The geometry
is represented as a discrete level set function, using vertex
values. Based on this
representation the cut cells and their corresponding quadrature rules are
constructed via the TPMC library \cite{tpmc}.

In figure \ref{fig: conv plot}, we show convergence results for ramp angles of $ \gamma=25^{\circ}$ and $\gamma=45^{\circ}$ in the $L^1$ and $L^{\infty}$ norm.
In the $L^1$ norm we observe convergence orders that are (roughly) $p+1$ for polynomials of degree $p$ for both angles. In the $L^{\infty}$ norm, the results are between $p+\frac 1 2$ and $p+1$ with less decay for even polynomial degrees. This is overall consistent with the findings of Giuliani \cite{Giuliani_DG} who reports for the annulus test for the $L^{\infty}$ error orders between $p + \frac 1 2$ and $p+1$ for polynomials of degrees $p=1,\ldots,5$.


\begin{figure}[ht]
    \includegraphics[width=0.49\linewidth]{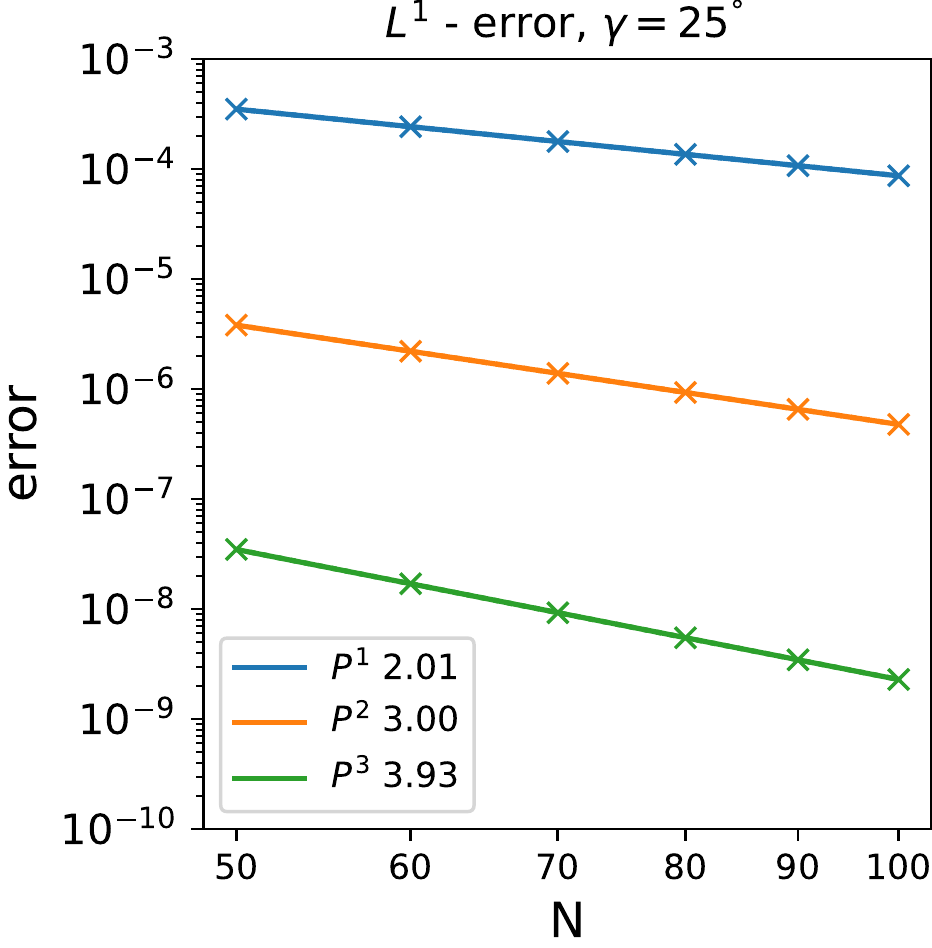}
    \hfill
    \includegraphics[width=0.49\linewidth]{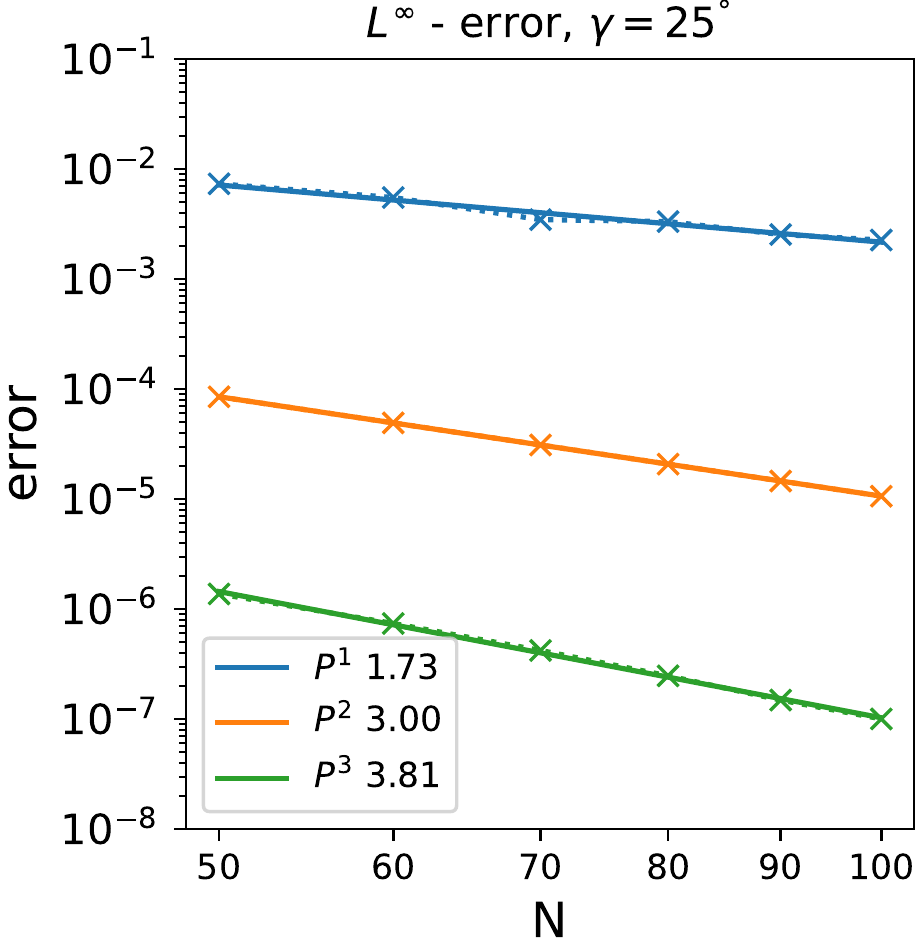}
    \includegraphics[width=0.49\linewidth]{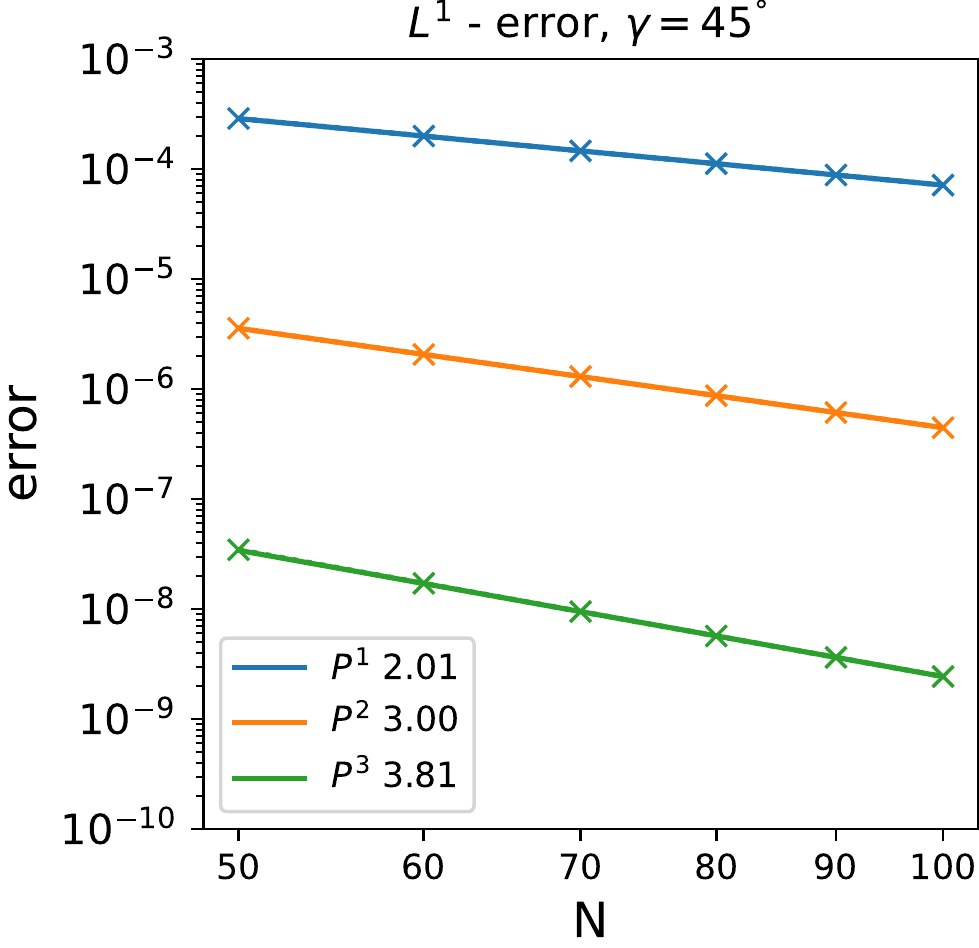}
    \hfill
    \includegraphics[width=0.49\linewidth]{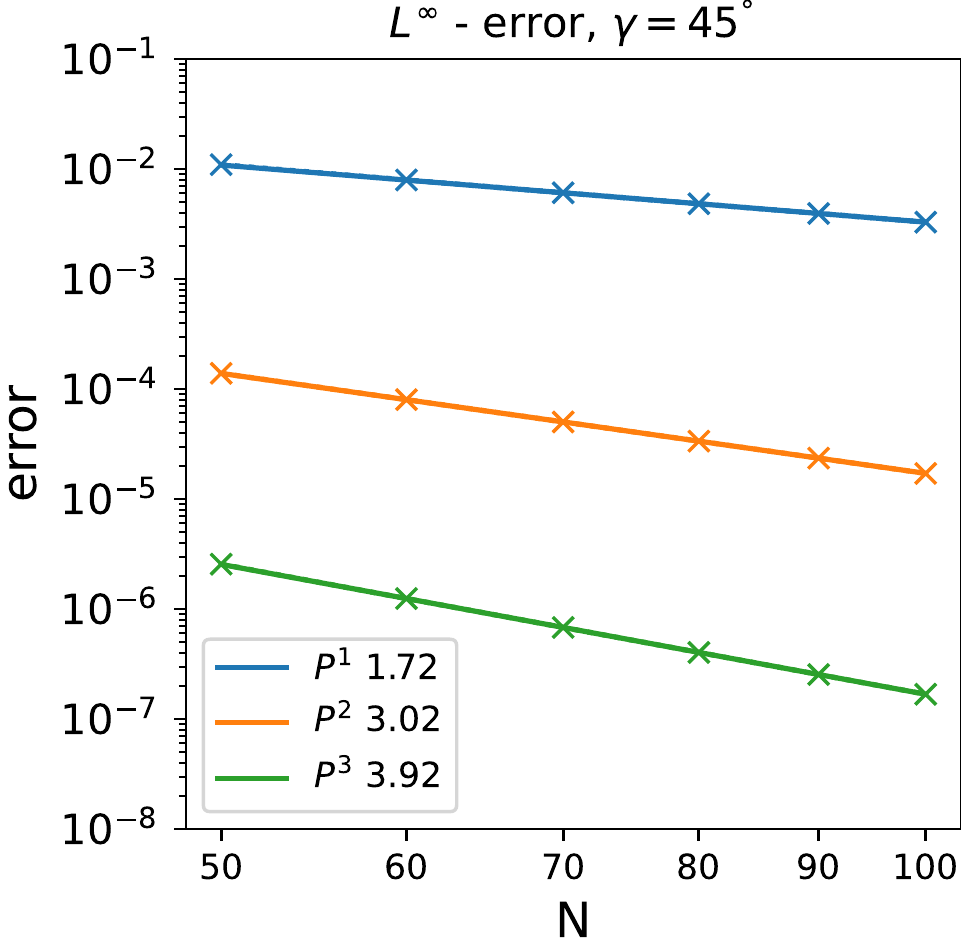}
    \caption{Convergence orders in $L^1$ and $L^\infty$ norm for the error at time $T=0.3$ for a ramp geometry with $\gamma = 25^\circ$ and $\gamma=45^\circ$ and different polynomial degrees $p=1,2,3$.}
    \label{fig: conv plot}
\end{figure}
\FloatBarrier
\section{Conclusion}
In this contribution, we introduce the formulation of the DoD stabilization for the linear advection equation for higher order polynomials. Compared to \cite{DoD_SIAM_2020}, where we only considered linear polynomials, we have augmented the penalty term $J_h^{1,E}$ to also involve the extended test function of the inflow neighbor of a small cut cell. For this new formulation, we show an $L^2$ stability result for the semi-discrete stabilized scheme for the ramp geometry. We also provide numerical results for a smooth test function, showing convergence rates between $p+\frac 1 2$ and $p+1$ for polynomial degree $p$. In the future, we plan to extend the stabilization to non-linear problems in 2d.
\FloatBarrier

\bibliographystyle{spmpsci}
\bibliography{references}

\end{document}